\documentclass[11pt]{article}
\usepackage{setting}

\title{Abundance of Unique Subhypergraphs }
\author{Xichao Shu\thanks{Institute of Mathematics, Leipzig University, Leipzig, Germany. Supported by the Alexander von Humboldt Foundation in the framework of the Alexander von Humboldt Professorship of Daniel Kr{\'a}l' endowed by the Federal Ministry of Education and Research. Email: \texttt{xichao.shu@uni-leipzig.de}. }\and
Zhuo Wu\thanks{Departament de Matemàtiques, Universitat Politècnica de Catalunya (UPC),
Carrer de Pau Gargallo 14, 08028 Barcelona, Spain. Z. Wu acknowledges the bilateral AEI+DFG research project PCI2024-155080-2: SRC-ExCo – Structure, Randomness and Computational Methods in Extremal Combinatorics, and the PID2023-147202NB-I00 (COCOA: COntemporary COmbinatorics and its Applications), all funded by MICIU/AEI/10.13039/501100011033. Email: \texttt{zhuo.wu@upc.edu}.}\and
Yisai Xue\thanks{School of Mathematics and Statistics, Ningbo University, Ningbo, China. Supported by the National Natural Science Foundation of China (No. 12501486). Email: xueyisai@nbu.edu.cn}}

\date{}

\begin{document}
\maketitle

\begin{abstract}
Given $k$-uniform hypergraphs $G$ and $H$, we say that $G$ is a unique subhypergraph of $H$ if $H$ contains exactly one subhypergraph isomorphic to $G$.
For an $n$-vertex $k$-graph $H$, let $f_k(H)$ be the number of non-isomorphic unique subhypergraphs of $H$, normalized by $2^{\binom n k}/n!$, and let $f_k(n)$ be the maximum of $f_k(H)$ over all $n$-vertex $k$-graphs $H$.
In the graph case $k=2$, Erd\H{o}s asked whether there exists a constant $\delta>0$ such that $f_2(n)>\delta$ for all $n$, offering \$100 for a proof and \$25 for a disproof.

Recently, Brada\v{c} and Christoph answered this question in the negative,, proving that $f_2(n)$ tends to $0$, or equivalently that no $n$-vertex graph contains a positive proportion of all $n$-vertex graphs as unique subgraphs. In this paper we show that the situation is fundamentally different for $k$-uniform hypergraphs with $k\ge3$.
In particular, for every fixed integer $k\ge 3$, we prove that
$\liminf_{n\to\infty} f_k(n) \ge 2/9$.

\end{abstract}

\section{Introduction}

A common feature of many extremal problems is that only the existence of a
prescribed subgraph is relevant, while the number of its copies plays no role. The problem studied in this paper is of a different nature.  We
are interested in graphs that appear uniquely in a host graph, meaning that the
host contains exactly one subgraph isomorphic to the given graph.

Given graphs $G$ and $H$, we call $G$ a unique subgraph of $H$, and write $G\subseteq^*H$, if $H$ contains exactly one subgraph isomorphic to $G$.  For an $n$-vertex graph $H$, define $f_2(H)$ to be the number of non-isomorphic graphs $G$ on at most $n$ vertices satisfying $G\subseteq^*H$, normalized by $2^{\binom n2}/n!$.  Let $f_2(n)$ be the maximum of $f_2(H)$ over all graphs $H$ on $n$ vertices.  The denominator is the right scale, since the number of non-isomorphic graphs on $n$ vertices is $(1+o(1))2^{\binom n2}/n!$, by the classical enumeration theory of P\'olya and Wright \cite{Polya1937,Wright1971,Wright1972}.

The study of this parameter goes back to Entringer and Erd\H{o}s \cite{EntringerErdos1972}, who proved the first general lower bound, namely $f_2(n)\ge e^{-cn^{3/2}}$ for some constant $c>0$ and all sufficiently large $n$. Later, Harary and Schwenk \cite{HararySchwenk1973} improved this to $f_2(n)\ge e^{-cn\log n}$, and Brouwer \cite{Brouwer1975} further improved it to $f_2(n)\ge e^{-cn}$.  However, these estimates still fall far short of a positive constant. Erd\H{o}s asked whether such a constant lower bound is possible: is there a $\delta>0$ such that $f_2(n)>\delta$ for every $n$?  He offered \$100 for a proof and \$25 for a disproof \cite{Erdos1976}. The larger prize for a proof reflects the strength of the positive statement: it would say that, despite the rigidity imposed by uniqueness, one can still realize a positive proportion of all isomorphism types uniquely inside a single host. Recently, Brada\v{c} and Christoph gave a negative answer to this question.

\begin{theorem}[Brada\v{c} and Christoph \cite{BradacChristoph2025}]\label{thm:BC}
 \[\lim_{n\to\infty} f_2(n)=0.\]
\end{theorem}

Theorem~\ref{thm:BC} gives a definitive answer to Erd\H{o}s's question in the graph case.  However, the problem itself is not tied to any feature special to edges of size two.  It only asks whether an isomorphism type is represented exactly once inside a host, a notion that carries over without change once graphs are replaced by uniform hypergraphs.  Thus the $k$-uniform setting is the most direct place to test whether the rarity of unique subgraphs is a general phenomenon.

We now formulate this analogue.  Throughout the paper, a $k$-graph means a $k$-uniform hypergraph.  Given $k$-graphs $G$ and $H$, we say that $G$ is a unique subhypergraph of $H$, again written $G\subseteq^*H$, if $H$ contains exactly one subhypergraph isomorphic to $G$.  For an $n$-vertex $k$-graph $H$, let $f_k(H)$ be the number of non-isomorphic $k$-graphs $G$ on at most $n$ vertices with $G\subseteq^*H$, divided by $2^{\binom nk}/n!$, and let $f_k(n)$ be the maximum of $f_k(H)$ over all $n$-vertex $k$-graphs $H$.

The natural question is then whether the graph case reflects a general phenomenon: for each fixed $k\ge2$, must $f_k(n)$ tend to zero as $n\to\infty$?  Somewhat surprisingly, our main result shows that the answer is no as soon as $k\ge3$.  Thus the rarity of unique subgraphs in Theorem~\ref{thm:BC} is a feature of graphs, not of unique containment itself.  More precisely, we prove the following.

\begin{theorem}\label{thm:main}
For every fixed integer $k\ge 3$,
\[
       \liminf_{n\to\infty} f_k(n) \ge \frac{2}{9}.
\]
\end{theorem}

The reason for the difference is a simple change of scale.  In graphs, even the smallest non-trivial vertex permutation moves only $\Theta(n)$ edges.  In $k$-graphs with $k\ge3$, it already moves $\Theta(n^{k-1})$ edges, while the sparse pattern used in our construction has only about $\log_2(n!)$ edges.

This gap is the source of the phenomenon. It allows us to delete a sparse pattern from the complete $k$-graph in such a way that the pattern barely overlaps its non-trivial vertex-permuted copies.  The resulting complement then has many unique subhypergraphs on the natural scale.

\paragraph{Paper organization.}
We begin in Section~2 with a reduction from
the original unlabelled problem to a probabilistic statement about unique
subhypergraphs in random labelled $k$-graphs.  The argument in Section~3 then
shows how the main theorem follows from the existence of a suitable sparse
$k$-graph with small average overlap.  This sparse pattern is constructed in
Section~4, where we also prove the required overlap estimate.  The paper ends
with a few concluding remarks in Section~5.

\paragraph{Notation, terminology and preliminaries.}
Throughout the paper, $k\ge3$ is fixed and $n$ tends to infinity.  All
asymptotic notation is with respect to $n\to\infty$; the implicit constants in
$O_k(\cdot)$ and $\Omega_k(\cdot)$ may depend on $k$.
We use $\exp(x)$ for $e^x$, and write
$\mathbf 1_A$ for the indicator of an event $A$. We write $[n]=\{1,\ldots,n\}$, $\mathcal T_n=\binom{[n]}k$, and $N=\binom nk$.  A labelled $k$-graph on $[n]$ is identified with a subset of $\mathcal T_n$, and $K_n^{(k)}$ denotes the complete $k$-graph on $[n]$.

If $\pi\in S_n$ and $F\subseteq\mathcal T_n$, then $\pi(F)$ denotes the $k$-graph obtained by applying $\pi$ to every edge of $F$.  The automorphism group of $F$ is denoted by $\mathrm{Aut}(F)$.  For a $k$-graph $F$, we write $v(F)$ and $e(F)$ for the numbers of vertices and edges of $F$.

An embedding $G\to H$ is an injective map $\phi:V(G)\to V(H)$ such that $\phi(e)\in E(H)$ for every $e\in E(G)$.  We write $G\to^*H$ if there is exactly one embedding from $G$ to $H$.  This is different from $G\subseteq^*H$, which means that there is exactly one unlabelled copy of $G$ in $H$.  The distinction matters only for $k$-graphs with non-trivial automorphisms.

Finally, $\cG^{(k)}(n,p)$ denotes the binomial random $k$-graph on vertex set
$[n]$, where each element of $\mathcal T_n$ is chosen independently with
probability $p$.  In particular, $\cG^{(k)}(n,1/2)$ is the uniform distribution
on labelled $k$-graphs on $[n]$.  We also use $(n)_t=n(n-1)\cdots(n-t+1)$ for the falling factorial. For an integer-valued random variable $X$, we write $(X)_2=X(X-1)$ for its second falling factorial, so that $\cE\bigl[(X)_2\bigr]$ is the factorial second
moment of $X$.

\section{A Probabilistic Reformulation}

Throughout this section, let $k\ge 3$ be fixed and put $N=\binom{n}{k}$. Let
$\mathcal U_k(n)$ be the set of unlabelled $k$-graphs on $n$ vertices, and let $\mathcal L_k(n)$ be the set of labelled $k$-graphs on vertex set $[n]$.
We use the following enumeration theorem for uniform hypergraphs.
\begin{theorem}[\cite{Qia14}]\label{thm:qian}
For any integer $k$ with $2\le k\le n/2$,  $|\mathcal U_k(n)|=(1+o(1))2^N/n!$.
\end{theorem}

The reduction below follows the argument of Brada\v{c} and Christoph~\cite{BradacChristoph2025} in the graph case.  For completeness, we present the details in the setting of $k$-uniform hypergraphs.

\begin{lemma}\label{lem:labelled-reduction}
Let $k\ge 3$.
For every $n$-vertex $k$-graph $H$, we have
\[
        f_k(H)=\cP_{G\sim\cG^{(k)}(n,1/2)}[G\to^*H]+o(1).
\]
\end{lemma}

\begin{proof}
    Let $\mathcal A_k(n)=\{G\in\mathcal U_k(n):|\operatorname{Aut}(G)|\ge 2\}$ be the set of unlabelled $k$-graphs on $n$ vertices with a non-trivial automorphism.
    Since an unlabelled $k$-graph $G$ has exactly $n!/|\operatorname{Aut}(G)|$ distinct labellings, we have
\[
\begin{aligned}
        2^N
        =|\mathcal L_k(n)|
        =\sum_{G\in\mathcal U_k(n)}
          \frac{n!}{|\operatorname{Aut}(G)|}   
        \le n!|\mathcal U_k(n)|
          -\frac{n!}{2}|\mathcal A_k(n)|.
\end{aligned}
\]
    Using Theorem \ref{thm:qian} and dividing by $n!$, it follows that $|\mathcal A_k(n)|=o(2^N/n!)$.

    The number of unlabelled $k$-graphs on at most $n-1$ vertices is at most $\sum_{i=0}^{n-1}2^{\binom{i}{k}}
        \le n2^{\binom{n-1}{k}}=o\left(\frac{2^N}{n!}\right)$.
        Using this and $|\mathcal A_k(n)|=o(2^N/n!)$, we have
    \[
\begin{aligned}
        f_k(H)
        &=
        \frac{n!}{2^N}
        \left|
        \{G\in\mathcal U_k(n):G\subseteq^* H\}
        \right|+o(1)                                                        \\
        &=
        \frac{n!}{2^N}
        \left|
        \{G\in\mathcal U_k(n):G\subseteq^* H
        \text{ and }|\operatorname{Aut}(G)|=1\}
        \right|+o(1)                                                        \\
        &=
        2^{-N}
        \left|
        \{G\in\mathcal L_k(n):G\subseteq^* H
        \text{ and }|\operatorname{Aut}(G)|=1\}
        \right|+o(1)                                                        \\
        &=
        \mathbb P_{G\sim \mathcal{G}^{(k)}(n,1/2)}
        \bigl[
        G\subseteq^* H
        \text{ and }|\operatorname{Aut}(G)|=1
        \bigr]+o(1).
\end{aligned}
\]
For an $n$-vertex $k$-graph $G$, each unlabelled copy of $G$ in $H$
    gives exactly $|\operatorname{Aut}(G)|$ embeddings of $G$ into $H$.
    Hence, for such $G$, $G\to^*H$ holds if and only if
    $G\subseteq^*H$ and $|\operatorname{Aut}(G)|=1$. This completes the proof.
\end{proof}

\section{Proof of Theorem \ref{thm:main}}
In this section, we state the sparse-overlap property needed for the
argument and show that it implies Theorem~\ref{thm:main}.  The existence
of a $k$-graph with this property will be proved in the next section.

\begin{proposition}\label{prop:overlap}
Let \(m\) be an integer such that \(m=\log_2(n!)+O(1)\). Then there exists an $m$-edge $k$-graph $F$ on $[n]$ such that
\[
        \frac1{n!}\sum_{\tau\in S_n,\ \tau\ne\id}
        2^{|F\cap\tau(F)|}
        \le 1+o(1).
\]
\end{proposition}

We now finish the proof of Theorem~\ref{thm:main}.  

\begin{proof}[Proof of Theorem \ref{thm:main}]
Let $m_0=\lceil\log_2 n!\rceil$, $m_1=m_0+1$. For $i=0,1$, fix an $m_i$-edge $k$-graph $F_i$ as in
Proposition~\ref{prop:overlap}, and let $H_i=K_n^{(k)}\setminus F_i$.
Let $G\sim\cG^{(k)}(n,1/2)$ and $R=K_n^{(k)}\setminus G$.  Note that the map $\phi\mapsto\phi^{-1}$ gives a bijection
between embeddings $G\to H_i$ and embeddings $F_i\to R$.  Indeed,
$\phi(G)\subseteq H_i$ holds precisely when no edge of $F_i$ lies in
$\phi(G)$, which is precisely the condition that $\phi^{-1}(F_i)\subseteq R$.
Let $X_i$ denote the number of labelled embeddings of $F_i$ in $R$.  Then
\[
        X_i=\sum_{\pi\in S_n}\mathbf 1_{\{\pi(F_i)\subseteq R\}}.
\]
Consequently, \(X_i\) is exactly the number of labelled embeddings from \(G\)
to \(H_i\). Hence \(G\to^* H_i\) holds precisely when \(X_i=1\), and therefore
\[
        \mathbb P_{G\sim\mathcal G^{(k)}(n,1/2)}[G\to^*H_i]
        =
        \mathbb P[X_i=1].
\]
Let $\lambda_i=\cE X_i=n!2^{-m_i}$.
Then we have $\lambda_0\in(1/2,1]$, $\lambda_1=\lambda_0/2\in(1/4,1/2]$.

Recall that $(X_i)_2=X_i(X_i-1)$.  Thus $\cE\bigl[(X_i)_2\bigr]$ is the expected number
of ordered pairs of distinct labelled copies of $F_i$ in $R$.  Hence
\[
        \cE\bigl[(X_i)_2\bigr]
        =
        \sum_{\pi\ne\sigma}
        \cP\bigl(\pi(F_i)\cup\sigma(F_i)\subseteq R\bigr)
        =
        \sum_{\pi\ne\sigma}2^{-|\pi(F_i)\cup\sigma(F_i)|}.
\]
Putting $\tau=\pi^{-1}\sigma$, we have $\tau\ne\id$, and for each fixed
$\tau$ there are $n!$ choices of $\pi$.  Also
\[
        |\pi(F_i)\cup\sigma(F_i)|
        =
        |F_i\cup\tau(F_i)|
        =
        2m_i-|F_i\cap\tau(F_i)|.
\]
Therefore
\[
        \cE\bigl[(X_i)_2\bigr]
        =
        n!2^{-2m_i}
        \sum_{\tau\ne\id}2^{|F_i\cap\tau(F_i)|}
        \le
        (1+o(1))n!^2 2^{-2m_i}
        =
        (1+o(1))\lambda_i^2,
\]
where the inequality follows from Proposition~\ref{prop:overlap}.

Since $X_i$ is integer-valued, $ \cP[X_i=1]\ge \cE X_i-\cE\bigl[(X_i)_2\bigr]$. 
Therefore, $\cP[X_i=1]\ge \lambda_i-(1+o(1))\lambda_i^2.$ Since $\lambda_0\in(1/2,1]$ and $\lambda_1=\lambda_0/2$, we get
\[
\max_{i=0,1}\cP[X_i=1]
\ge
\max\left\{
        \lambda_0-\lambda_0^2,\,
        \frac{\lambda_0}{2}-\frac{\lambda_0^2}{4}
\right\}
-o(1).
\]
The minimum of the last maximum over $\lambda_0\in(1/2,1]$ is $2/9$, attained when
$\lambda_0=2/3$.  Hence, for this choice of $i=i(n)$, $ \cP[X_i=1]\ge \frac{2}{9}-o(1).$ By Lemma~\ref{lem:labelled-reduction},
\[
        f_k(H_i)=\cP_{G\sim\cG^{(k)}(n,1/2)}[G\to^*H_i]+o(1)
        =\cP[X_i=1]+o(1)
        \ge \frac{2}{9}-o(1).
\]
Thus
\[
        f_k(n)\ge f_k(H_i)\ge \frac{2}{9}-o(1).
\]
\end{proof}


\section{Sparse $k$-Graphs with Small Overlap}
In this section we prove Proposition~\ref{prop:overlap}.  We first record two
elementary estimates for random subsets chosen without replacement.  The first
is a direct consequence of Maclaurin's inequality~\cite{HLP52} for elementary symmetric means.

\begin{lemma}\label{lem:hypergeom}
Let \(R\) be the number of marked elements in a uniformly chosen \(m\)-element
subset of an \(N\)-element set with \(D\) marked elements.  Set \(q=D/N\).
Then, for every \(0\le a\le1\),
\[
        \cE \bigl[a^R\bigr]\le (1-q+qa)^m.
\]
\end{lemma}

\begin{proof}
Consider the \(N\) weights consisting of \(D\) copies of \(a\) and \(N-D\)
copies of \(1\).  Then \(\cE \bigl[a^R\bigr]\) is the \(m\)th elementary symmetric mean of
these weights.  By Maclaurin's inequality, this is at most the \(m\)th power of
their arithmetic mean, namely
\[
        \left(\frac{Da+N-D}{N}\right)^m=(1-q+qa)^m .
\]
\end{proof}

The second estimate says that a sparse random set has few adjacencies along a permutation without fixed points.

\begin{lemma}\label{lem:cycle-sampling}
There is an absolute constant $C$ such that the following holds.  Let $\rho$ be
a fixed-point-free permutation of a finite set $\Omega$ with $|\Omega|=D$.  Let
$B$ be a uniformly random $r$-element subset of $\Omega$, where $r\le D/2$.
Then
\[
        \cE \bigl[2^{|B\cap \rho(B)|}\bigr]
        \le \exp\left(C\frac{r^2}{D}\right).
\]
\end{lemma}

\begin{proof}
Let $L$ be the graph on vertex set $\Omega$ obtained by joining each
$x\in\Omega$ to $\rho(x)$ and then forgetting orientations and multiple edges.
Since $\rho$ has no fixed points, the graph $L$ has no loops.  Moreover,
\(\Delta(L)\le2\), because each vertex is adjacent only to its predecessor and
successor in its \(\rho\)-cycle.

Let  \(e_L(B)\) denote the number of edges of \(L\) with both endpoints in
\(B\). Then $|B\cap\rho(B)|\le 2e_L(B)$. Hence, it is enough to prove that
\[
        \cE\bigl[4^{e_L(B)}\bigr]
        \le \exp\left(O\left(\frac{r^2}{D}\right)\right).
\]
Write \(p=r/D\).  For every fixed subset \(S\subseteq\Omega\), we have
\[
        \cP(S\subseteq B)
        =
        \frac{(r)_{|S|}}{(D)_{|S|}}
        \le p^{|S|}.
\]
Expanding
\[
        4^{e_L(B)}
        =
        \prod_{uv\in E(L)}
        \left(1+3\cdot\mathbf 1_{\{u,v\}\subseteq B}\right),
\]
we obtain
\[
\begin{aligned}
        \cE\bigl[4^{e_L(B)}\bigr]
        &=
        \sum_{A\subseteq E(L)}
        3^{|A|}
        \cP\bigl(V(A)\subseteq B\bigr)  
        &\le
        \sum_{A\subseteq E(L)}
        3^{|A|}p^{v(A)},
\end{aligned}
\]
where \(V(A)\) is the set of vertices incident with the edges of \(A\), and
\(v(A)=|V(A)|\).  Put
\[
        S_L:=\sum_{A\subseteq E(L)}3^{|A|}p^{v(A)}.
\]
It remains to prove that \(S_L\le \exp(O(Dp^2))\).

We now bound the last sum by considering different ranges of p.

\paragraph{Case 1: $p\le 1/6$.}

Since \(L\) has maximum degree at most two, its components are paths and
cycles. Moreover, the sum
\[
        \sum_{A\subseteq E(L)}3^{|A|}p^{v(A)}
\]
factors over the connected components of \(L\), so it suffices to consider each  component separately.

If \(J\) is a path with \(t\) vertices, then for any non-empty \(A\subseteq E(J)\),
write \(A\) as a union of maximal consecutive blocks of selected edges.  A
block of length \(\ell\) has weight \(3^\ell p^{\ell+1}\).  The total weight
of all possible blocks in \(J\) is at most
\[
        t\sum_{\ell\ge1}3^\ell p^{\ell+1}
        =
        O(tp^2),
\]
because \(p\le 1/6\).  Dropping the requirement that the chosen blocks be
disjoint can only increase the sum, and hence
\[
        \sum_{A\subseteq E(J)}3^{|A|}p^{v(A)}
        \le
        \sum_{s\ge0}\frac{1}{s!}
        \left(
        O(tp^2)
        \right)^s
        =
        \exp(O(tp^2)).
\]

The same argument applies when \(J\) is a cycle, except for the case where all
edges of \(J\) are selected.  If \(J\) has \(t\) vertices, that exceptional
choice contributes \((3p)^t\), which is \(O(tp^2)\) for \(t\ge3\) and
\(p\le1/6\).  Thus the same bound
\[
        \sum_{A\subseteq E(J)}3^{|A|}p^{v(A)}
        \le \exp(O(tp^2))
\]
holds for cycles as well. Multiplying over all components \(J\) of \(L\), and using
\(\sum_J |V(J)|=D\), we get
\[
        S_L
        \le
        \prod_J \exp(O(|V(J)|p^2))
        =
        \exp(O(Dp^2)).
\]

\paragraph{Case 2: $p> 1/6$.}

  Since $e_L(B)\le r$, we have
\[
        4^{e_L(B)}\le 4^r=\exp(O(r))=\exp(O(Dp)).
\]
As $p>1/6$, we have $Dp=O(Dp^2)$.  Thus in this case as well,
\[
        \cE\bigl[4^{e_L(B)}\bigr]
        \le
        \exp(O(Dp^2)).
\]

Combining the two cases and recalling that $Dp^2=r^2/D$, we obtain
\[
        \cE\bigl[2^{|B\cap \rho(B)|}\bigr]
        \le
        \cE\bigl[4^{e_L(B)}\bigr]
        \le
        \exp\left(C\frac{r^2}{D}\right)
\]
for some absolute constant $C$.  This completes the proof.
\end{proof}

We are now prepared to prove the main proposition.

\begin{proof}[Proof of Proposition \ref{prop:overlap}]
Choose $F$ uniformly among all $m$-subsets of $\cT_n$.  We prove that the expected value of the displayed average is at most $1+o(1)$.

Fix $\tau\ne\id$.  Let
\[
        \Omega_\tau=\{e\in\cT_n:\tau(e)\ne e\},
        \qquad
        D=D(\tau)=|\Omega_\tau|,
        \qquad
        q=\frac{D}{N}.
\]
Thus $\Omega_\tau$ is the set of $k$-sets moved by $\tau$.

\begin{claim}\label{claim:single-permutation}
For any $\tau\ne\id$,
\[
        \cE_F\bigl[2^{|F\cap\tau(F)|}\bigr]
        \le
        (1+o(1))(2-q)^m .
\]
\end{claim}

\begin{poc}
We first note that \(D=\Omega_k(n^{k-1})\) uniformly over all \(\tau\ne\id\). Indeed, choose a vertex $x$ with $\tau(x)\ne x$,
and put $y=\tau(x)$.  Every $k$-set containing $x$ but not $y$ is moved by
$\tau$, giving at least $\binom{n-2}{k-1}$ moved $k$-sets.

Let $B=F\cap\Omega_\tau, R=|B|.$
Then $R$ has the hypergeometric distribution obtained by choosing $m$ elements
from an $N$-element set with $D$ marked elements.  Conditional on $R=r$, the set
$B$ is a uniformly random $r$-element subset of $\Omega_\tau$.  Moreover, the
restriction of $\tau$ to $\Omega_\tau$ has no fixed point.

Since $\Omega_\tau$ is $\tau$-invariant and every $k$-set outside
$\Omega_\tau$ is fixed setwise by $\tau$, we have
\[
        |F\cap\tau(F)|
        =
        m-R+|B\cap\tau(B)|.
\]

Since $m=O(n\log n)$ and $D=\Omega_k(n^{k-1})$, we have $m\le D/2$ for all
large $n$.  By Lemma~\ref{lem:cycle-sampling},
\[
        \cE\left[
        2^{|F\cap\tau(F)|}\mid R=r
        \right]
        \le
        2^{m-r}\exp\left(C\frac{r^2}{D}\right).
\]
Taking expectation over $R$, we get
\[
        \cE_F\bigl[2^{|F\cap\tau(F)|}\bigr]
        \le
        \cE_R\left[
        2^{m-R}\exp\left(C\frac{R^2}{D}\right)
        \right].
\]

Using $R\le m$, we have $R^2/D\le mR/D$.  Hence
\[
        2^{m-R}\exp\left(C\frac{R^2}{D}\right)
        \le
        2^m\left(2^{-1}e^{Cm/D}\right)^R.
\]
Set $a_\tau=2^{-1}e^{Cm/D}$. Then $a_\tau\le1$ for all
large $n$.  Applying Lemma~\ref{lem:hypergeom} to $R$, we have
\[
\begin{aligned}
        \cE_F\bigl[2^{|F\cap\tau(F)|}\bigr]
        \le
        2^m\,\cE_R \bigl[a_\tau^R\bigr]  
        \le
        2^m\left(1-q+q a_\tau\right)^m  
        =
        2^m\left(1-q+\frac{q}{2}e^{Cm/D}\right)^m .
\end{aligned}
\]

Let $\delta_\tau= q\left(e^{Cm/D}-1\right).$ Since $e^{Cm/D}=1+O(m/D)$ and $q=D/N$, we have $\delta_\tau= O\left(\frac{m}{N}\right).$ Since $m=O(n\log n)$ and $N=\binom nk=\Theta_k(n^k)$ with $k\ge3$,  we have $m^2/N=o(1)$. Hence

\[\begin{aligned}
       \left(2-2q+qe^{Cm/D}\right)^m
 &=
         (2-q)^m
        \left(1+\frac{\delta_\tau}{2-q}\right)^m  \\
        &\le
        (2-q)^m \exp(m\delta_\tau)  \\
        &=
        (2-q)^m \exp\left(O\left(\frac{m^2}{N}\right)\right)\\
        &= (1+o(1))(2-q)^m.
\end{aligned}\]

This completes the proof.
\end{poc}

By linearity of expectation and Claim~\ref{claim:single-permutation},
\[
        \cE_F\left[
        \frac1{n!}\sum_{\tau\in S_n,\ \tau\ne\id}
        2^{|F\cap\tau(F)|}
        \right]
        \le
        (1+o(1))
        \frac1{n!}\sum_{\tau\in S_n,\ \tau\ne\id}
        (2-q(\tau))^m .
\]
It remains to show that this last average is at most $1+o(1)$.

We shall use the following simple observation.  If a setwise fixed \(k\)-set
contains a moved vertex \(x\), then it contains the whole non-trivial cycle of
\(\tau\) containing \(x\).  Hence its intersection with the moved vertices is a
union of non-trivial cycles, and is either empty or has size at least two.

Let \(s\) be the number of moved vertices of \(\tau\), and let \(f=n-s\). The number of \(k\)-sets contained entirely in the fixed-vertex set
is exactly \(\binom fk\), which is at most $\left(\frac fn\right)^k\binom n k$. If a setwise fixed $k$-set contains exactly $\ell$ moved
vertices where $2\le \ell\le k$, then these $\ell$ vertices must be a union of
non-trivial cycles.  The number of ways to choose these cycles is at most
\(O_k(s^{\lfloor \ell/2\rfloor})\), and the remaining $k-\ell$ vertices can be chosen in at most
$O_k(f^{k-\ell})$ ways. Dividing by $\binom nk=\Theta_k(n^k)$ and summing over $2\le \ell\le k$, we get
\[
        1-q
        \le
\left(\frac fn\right)^k
        +
        O_k\left(
        \sum_{\ell=2}^k
        \frac{s^{\lfloor \ell/2\rfloor}f^{k-\ell}}{n^k}
        \right).
\]

We now bound the last average by splitting into two ranges according to \(s\).

\paragraph{Case 1:  $2\le s\le n/2$.}
 Since $s,f\le n$, $O_k(s^{\lfloor\ell/2\rfloor}f^{k-\ell})\le O_k(s n^{k-2}).$
 Consequently,
\[
        q\ge 
        1-\left(1-\frac sn\right)^k
        -O_k\left(\frac{s}{n^2}\right).
\]
   
There are at most $n^s$ permutations moving exactly $s$ vertices. Also, $ (2-q)^m
        =
        2^m(1-q/2)^m
        \le
        2^m\exp(-mq/2),$
and $2^m= O(n!)$.  Thus the contribution of this range to the average is at most
a constant times
\[
        \sum_{2\le s\le n/2}
        n^s
        \exp\left[
        -\frac m2
        \left(
        1-\left(1-\frac sn\right)^k
        -O_k\left(\frac{s}{n^2}\right)
        \right)
        \right].
\]

Since $k\ge 3$, there exists a constant $\eta=\eta(k)>0$ such that $ 1-(1-\alpha)^k\ge (2\log 2+\eta)\alpha$ for every $0<\alpha\le 1/2$.  
Applying this with $\alpha=s/n$, we get for any $2\le s\le n/2$,
\[
        1-\left(1-\frac sn\right)^k
        \ge
        (2\log 2+\eta)\frac{s}{n}.
\]
Moreover,
\(
        O_k\left(\frac{s}{n^2}\right)
        =
        o\left(\frac{s}{n}\right)
\).  Hence, for all sufficiently large $n$,
\[
        1-\left(1-\frac sn\right)^k
        -
        O_k\left(\frac{s}{n^2}\right)
        \ge
        \left(2\log 2+\frac{\eta}{2}\right)\frac{s}{n}.
\]

Since 
\(
        m=\left(1+o(1)\right)\frac{n\log n}{\log 2},
\)
there exists a constant $\gamma=\gamma(k)>0$ such that the $s$th summand is at
most
\[
\begin{aligned}
        n^s
        \exp\left[
        -\frac m2
        \left(2\log 2+\frac{\eta}{2}\right)\frac{s}{n}
        \right]
        \le
        n^s\exp\left[-(1+\gamma)s\log n\right]  
        =
        n^{-\gamma s}.
\end{aligned}
\]
Therefore
\[
        \sum_{2\le s\le n/2}
        n^s
        \exp\left[
        -\frac m2
        \left(
        1-\left(1-\frac sn\right)^k
        -O_k\left(\frac{s}{n^2}\right)
        \right)
        \right]
        \le
        \sum_{s\ge 2} n^{-\gamma s}
        =
        o(1).
\]
This proves that the contribution from the range $2\le s\le n/2$ is $o(1)$.

\paragraph{Case 2:  $ s>n/2$.}
Let $p_{n,f}$ be the proportion of permutations of $[n]$ with exactly $f$
fixed points.  Since $s\le n$, the contribution of the range $f<n/2$ is therefore at most
\[
        (1+o(1))\sum_{0\le f<n/2}
        p_{n,f}
        \left(
        1+
        \left(\frac fn\right)^k
        +
        O_k\left(
        \sum_{\ell=2}^k
        \frac{n^{\lfloor \ell/2\rfloor}f^{k-\ell}}{n^k}
        \right)
        \right)^m.
\]
For each fixed $f$, we have $p_{n,f}\to e^{-1}/f!$. Choose \(C_k>0\) large enough and set
\[
        a_{n,f}:=
        \left(\frac fn\right)^k
        +
        C_k
        \sum_{\ell=2}^k
        \frac{n^{\lfloor \ell/2\rfloor}f^{k-\ell}}{n^k}.
\]

Then it suffices to show that
\[
\sum_{f=0}^{n/2}p_{n,f}(1+a_{n,f})^m\le 1+o(1).
\]

 We next split the expression above into three terms and estimate them separately. 
 We first choose \(\varepsilon>0\) sufficiently small, and then choose \(M\) sufficiently large but fixed independently of \(n\).

\paragraph{The range $f\le M$.} Since $k\ge 3$, $n^{\lfloor k/2\rfloor}\le n^{k-2}$, hence

\[
        a_{n,f}\le O\left(\frac{1}{n^3}\right)+
        O_k\left(
        \sum_{\ell=2}^k
        \frac{n^{\lfloor k/2\rfloor}}{n^k}
        \right)\le  O\left(\frac{1}{n^2}\right).
\]

Since $m=O(n\log n)$ and $\log(1+x)\le x$ for $x\ge0$, we have $ m\log(1+a_{n,f})\le m a_{n,f}=o(1)$, and hence 
\[
(1+a_{n,f})^m =\exp\bigl(m\log(1+a_{n,f})\bigr)\le \exp\bigl(ma_{n,f}\bigr)=1+o(1).
\] 
Consequently, 
\[
      \sum_{f=0}^M p_{n,f}(1+a_{n,f})^m
        =
        e^{-1}\sum_{f=0}^M \frac1{f!}+o(1)
        \le 1+o(1).
\]

\paragraph{The range $M<f\le \varepsilon n$.}
We first estimate the two terms of $m a_{n,f}$.  The contribution of
the term $m(f/n)^k$ is
\[
        m\left(\frac fn\right)^k
        =
        O_k\left(
        f\log f\cdot
        \left(\frac fn\right)^{k-1}
        \frac{\log n}{\log f}
        \right)
        =
        \bigl(O_k(\varepsilon^{k-1})+o(1)\bigr)f\log f.
\]
The last equality holds for $M<f\le \varepsilon n$: if
$f\le\sqrt n$, then the displayed ratio is $o(1)$, while if $f>\sqrt n$, then
$\log n/\log f\le2$ and $(f/n)^{k-1}\le\varepsilon^{k-1}$.

It remains to control the error terms in $ma_{n,f}$. For each $2\le \ell\le k$,
\[
        m\frac{n^{\lfloor \ell/2\rfloor}f^{k-\ell}}{n^k}
        =
        O_k\left(
        f\log f\cdot
        n^{\lfloor \ell/2\rfloor-k+1}
        f^{k-\ell-1}
        \frac{\log n}{\log f}
        \right)
        =
        o(f\log f)
\]
for $M<f\le \varepsilon n$.  Hence
\[
        m a_{n,f}
        \le
        \bigl(O_k(\varepsilon^{k-1})+o(1)\bigr)f\log f.
\]
Choosing $\varepsilon>0$ sufficiently small, we may assume that $m a_{n,f}\le \frac12 f\log f$ for all sufficiently large $n$ throughout this range. Therefore, using $p_{n,f}\le 1/f!$ and Stirling's formula,
\[
        p_{n,f}(1+a_{n,f})^m
        \le
        \exp\left(-(1+o(1))f\log f+\frac12 f\log f\right)
        \le
        \exp\left(-\frac13 f\log f\right)
\]
for all sufficiently large $f$.  
Consequently,
\[
\limsup_{n\to\infty}
\sum_{M<f\le\varepsilon n}
p_{n,f}(1+a_{n,f})^m
\le
\sum_{f>M}
\exp\left(-\frac13 f\log f\right).
\]
The right-hand side tends to $0$ as $M\to\infty$.  Thus, after choosing $M$
sufficiently large, the contribution from the range $M<f\le\varepsilon n$ is
$o(1)$.

\paragraph{The range \(\varepsilon n\le f<n/2\).}
Write \(f=\beta n\), where \(\varepsilon\le\beta<1/2\).  Then
\[
        a_{n,f}
        =
        \beta^k
        +
        O_k\left(
        \sum_{\ell=2}^k
        \beta^{k-\ell}n^{\lfloor \ell/2\rfloor-\ell}
        \right)
        =
        \beta^k+o(1).
\]
Using \(p_{n,f}\le1/f!\) and Stirling's formula,
\[
        \log p_{n,f}
        \le
        -f\log f+O(f)
        =
        -\beta n\log n+O(n).
\]
On the other hand,
\[
        m\log(1+a_{n,f})
        =
        \left(\log_2(1+\beta^k)+o(1)\right)n\log n.
\]
Since \(k\ge3\) and \(0\le\beta\le1/2\), we have $ \log_2(1+\beta^k)\le \beta^2$, therefore
\[
        \log\left(p_{n,f}(1+a_{n,f})^m\right)
        \le
        \left(-\beta+\beta^2+o(1)\right)n\log n.
\]
As \(\varepsilon\le\beta<1/2\), the coefficient \(-\beta+\beta^2\) is at most
\(-\varepsilon/2\).  Hence, for all sufficiently large \(n\),
\[
        p_{n,f}(1+a_{n,f})^m
        \le
        \exp\left(-\frac{\varepsilon}{3}n\log n\right).
\]
Thus
\[
        \sum_{\varepsilon n\le f<n/2}
        p_{n,f}(1+a_{n,f})^m
        \le
        n\exp\left(-\frac{\varepsilon}{3}n\log n\right)
        =
        o(1).
\]

Combining the three ranges, and then choosing \(M\) sufficiently large,
gives
\[
        \sum_{0\le f<n/2}p_{n,f}(1+a_{n,f})^m\le 1+o(1).
\]
Hence the contribution from Case 2 is at most \(1+o(1)\).
Together with Case 1, this proves that the expected average is at most
\(1+o(1)\), and therefore there exists a choice of \(F\) satisfying
Proposition~\ref{prop:overlap}.

\end{proof}

\section{Concluding remarks}
The proof isolates the difference in scale that distinguishes the hypergraph
setting from the graph case.  In fixed uniformity \(k\), a transposition moves $2\binom{n-2}{k-1}=\Theta(n^{k-1})$ \(k\)-edges, whereas the sparse pattern used in the construction has only
\(\Theta(n\log n)\) edges.  Thus the separation needed for the overlap estimate
is available for \(k\ge3\), but disappears when \(k=2\).

We have not attempted to optimize the constant.  The value $2/9$ comes from the
second-moment argument, together with the rounding in the choice of the number
of edges in the sparse pattern.  Along any subsequence for which one can choose
$m$ with $n!2^{-m}=1/2+o(1)$, the same argument would give the value $1/4$,
which is the best possible constant obtainable from this particular
second-moment lower bound.

It is natural to expect that the relevant copy-counting random variables should
exhibit Poisson-type behaviour.  Under such a heuristic, the probability of
having a unique copy is maximized at $1/e$, and this suggests that the optimal
constant may be $1/e$.  At present, however, we do not even know whether the limiting behaviour of $f_k(n)$ depends on $k$, and resolving this would require new ideas
beyond the single-pattern second-moment argument used here.

\section*{Acknowledgements}
The authors are grateful to Daniel Kr\'al' for reading an earlier version of the
paper and for several helpful comments.  The authors also acknowledge the use of
ChatGPT (GPT-5.5 Pro) in the early stage of this project.  In particular, the
main idea behind the construction in Proposition~\ref{prop:overlap} for the
case $k=3$ was first suggested during our interaction with ChatGPT.  The authors
subsequently verified the argument, refined the construction, and extended it
to the general $k$-uniform setting.  The authors assume full responsibility for
all statements and proofs presented in the paper.
\bibliographystyle{abbrv}
\bibliography{reference}

\end{document}